\documentclass[12pt,reqno,openany]{article}
\usepackage{amsmath,amssymb,amsthm,textcomp}
\usepackage{colortbl}
\usepackage[english]{babel}
\usepackage{dsfont}
\newcounter{ass}

\RequirePackage{amsmath}
\RequirePackage{amssymb}
\def\O{\Omega}\def\o{\omega}
\def\s{\sigma}
\def\a{\alpha}
\def\b{\beta}
\def\l{\lambda}
\def\d{\delta}
\def\D{\Delta}
\def\eps{\varepsilon}
\def\g{\gamma}\def\G{\Gamma}
\def\disp{\displaystyle}
\def\r{\rho}

\def\f{\varphi}
\def\th{\theta}
\def\E{{\bf E}}
\def\Prb{{\bf P}}\def\P{{\bf P}}
\def\tl{\widetilde}
\def\bar{\overline}
\def\td{\longrightarrow}
\def\N{\mathbf{N}}

\def\R{\mathbf{R}}

\newtheorem{theorem}{Theorem}[section]
\newtheorem{definition}[theorem]{Definition}
\newtheorem{proposition}[theorem]{Proposition}
\newtheorem{lemma}[theorem]{Lemma}
\newtheorem{corollary}[theorem]{Corollary}
\newtheorem{rem}[theorem]{Remark}
\numberwithin{equation}{section}
\begin{document}
\title{Finite  and  infinite  time  horizon   for BSDE with Poisson jumps.}
\author{Ahmadou Bamba Sow$^{(*)}$}\date{\today}
\maketitle \noindent{\footnotesize
 $^{(*)}$ LERSTAD, UFR de Sciences Appliqu\'ees et de Technologie,
Universit\'e Gaston Berger, BP 234, Saint-Louis, SENEGAL.  email :
{\tt
ahmadou-bamba.sow@ugb.edu.sn}}\\
\begin{abstract}
This   paper   is  devoted  to  solving  a real  valued  backward stochastic differential equation  with   jumps   where   the time horizon  may  be   finite  or  infinite. Under  linear  growth   generator, we   prove  existence of  a minimal  solution. Using a   comparison  theorem    we  show    existence  and  uniqueness of solution  to  such equations  when  the  generator  is   uniformly continuous  and satisfies  a  weakly  monotonic  condition.
\end{abstract}
Dol\'eans Dade  exponential.\\

%
\section{Introduction} 

After   the  pioneer work  of  Pardoux   and  Peng \cite{Par-Peng90} on   linear  Backward stochastic differential  equation (BSDE  in  short)  with   Lipschitz  generator,   the   interest   in   such  stochastic   equations     has  increased  thanks   to the   many   domains of applications   including stochastic  representation  of solutions  of   partial  differential   equations (PDEs in short).   For example, Pardoux and
Peng \cite{Par-Peng92}  and   Peng \cite{Peng1} proved that BSDEs provide a probabilistic formula for solutions of  
 quasilinear parabolic PDEs.

BSDEs with Poisson Process (BSDEP  in short) were first discussed by Tang and Li \cite{Tang-Li} and Wu \cite{Wu}. Studying
such equations, Barles {\it et al}  \cite{Bar-Buc-Par}  generalized the result in  \cite{Par-Peng92}, and obtained
a probabilistic interpretation of a solution of a parabolic integral-partial differential equation (PIDE). This   was done  by  means of   a  real-valued BSDEP  with       Lipschitzian   generator. Since then   many  efforts  have  been   done  in  relaxing  the  Lipschitz   assumption  of  the    generator of  the   BSDEs  (see  \cite{Bah, Kob, Lep-San, Mao} among others) and  the  BDSEP (see  \cite{Par4, Royer, Situ, Yin-Situ}).  
In \cite{Par4},  the   author solved   a   multidimensional  BSDEP  and   showed  an   existence  result  under   monotonicity  in  the  second  variable of  the drift  and   Lipschitz   condition   in the  other ones. Royer \cite{Royer}   focused  in  weakening  the   Lipschitz  condition  required on  the  last  variable  of  the  generator and    improved upon   the    results  given  in  \cite{Bar-Buc-Par}. The  key  point is  a strict  comparison  theorem  and    a  representation  of solution of   the  one dimensional    BSDEP  in  terms  of  non-linear  expectation. But   all  these   results  are   established  with  a   fixed  time   horizon $T$.  A  natural   question   is  under  which  condition   on  the   coefficients  the stochastic equation   still  has   a   solution   given   a  square  integrable    terminal  value  $\xi$ ? 
In  fact  this   problem   has   been  investigated   by  Peng   \cite{Peng1}  and   Darling and   Pardoux \cite{Dar-Par} and   others  researchers    when the   terminal value  $\xi$ is  null      or   satisfies  the   integrability  condition  $\E(e^{\l T} \xi^2)< \infty$,    for  some   $\l>0$  and   random  terminal  time $T$.   Chen  and Wang \cite{Che-Wan}   established   the   first  existence  and   uniqueness of solution      to  BSDE  with   infinite   time   horizon   when   the   generator  satisfies    a Lipschitz  type   condition.    Recently  Fan  {\it et  al}  \cite{She-Lon-Tian}  weakened   assumptions  required  in   \cite{Che-Wan}  and  prove   an   existence   and  uniqueness   result  under    mild   conditions  of  the   generator with  finite  or infinite  time   horizon.  

The aim  of  this   paper   is   to    extend the result  established in  \cite{She-Lon-Tian}   to  the  case   of   BSDEP. Our motivation comes  from the recent work  of Yao \cite{Yao}. The author proves  an existence and uniqueness result  of BSDEP with  infinite time interval  and some  monotonicity   condition  stronger than  those   in \cite{She-Lon-Tian}. In   this work  we show that the  results  obtained   in  \cite{She-Lon-Tian} 
can be extended to  BSDEP. 
The  paper   is  organized   as   follows.  We  first   prove  existence  of   a   minimal  solution  in   Section  2  and   a    comparison  theorem   in Section 3. 
Thanks  to  these   statements we deal   with  the   solvability  of finite  or   infinite   BSDEP   in Section 4. 
\section{BSDE with Poisson Jumps}
\subsection{Definitions and  preliminary  results}
Let $\O$ be   a   non-empty set,  ${\cal F}$   a   $\s-$algebra of sets  of $\O$  and   $\Prb$  a   probability  measure defined  on  $\cal F$.   The  triplet
 $(\O, {\cal F},\;  \Prb)$  
defines    a  probability  space, which   is  assumed    to  be  complete.  
We are    given     two   mutually  independent processes : 
\begin{itemize}
\item[$\bullet$] a $d-$dimensional  Brownian motion $(W_t)_{t\ge 0}$, 
\item[$\bullet$] a random Poisson   measure $\mu$ on  $E \times   \R_+$ with compensator $\nu(dt, de)  =  \l(de) dt$
\end{itemize}
where   the  space   $E = \R-\{0\}$ is equipped with  its Borel   field ${\mathcal E}$   such  that $\{\tl \mu ([0, t]\times  A)= (\mu - \nu)[0, t]\times  A \}$ is  a martingale    for   any  $A\in {\mathcal E}$  satisfying  $\l(A)<\infty$.  $\l$ is a $\s-$finite  measure  on ${\mathcal E}$ and    satisfies
$$ \int_E  (1\wedge|e|^2) \l(de) < \infty.$$
We    consider    the   filtration    $({\cal F}_{t})_{t\ge 0}$    given   by   $  {\cal F}_{t}   =   {\cal F}^W_{t} \vee    {\cal F}^\mu_{t}$,  
where    for any  process   $\{\eta_t\}_{t\ge0},  \; {\cal F}^\eta_{s,t} = \sigma\{\eta_r-\eta_s,  \; s\le r\le t\}\vee {\cal N},  \; \;  {\cal F}^\eta_{t} = {\cal F}^\eta_{0,t} $. $\cal N$ denotes  the  class  of   $\Prb-$null sets   of  $\cal F$. 

\noindent For $Q \in  \N^*, \; | \;.\; | $  stands  for the euclidian norm  in $\R^Q$.\\
We consider the following sets (where  $\E$  denotes   the   mathematical expectation   with respect to  the  probability  measure $\P$),  and   a  non-random horizon   time  $0<T\le + \infty$:
\begin{description}
\item[$\bullet$]
$\disp {S}^2 (\R^Q) $ the space of  ${\cal F}_t -$adapted c\`adl\`ag  processes
$$\Psi: [0, T]\times \O \td \R^Q, \;  \left\VertÊ\Psi  \right\Vert^2_{S^2(\R^Q)}  = \E \left(\sup_{0\leq t\leq T} |\Psi_t|^2 \right) < \infty.$$
\item[$\bullet$] $\disp H^2 (\R^Q) $ the space of ${\cal F}_t-$progressively measurable processes 
$$\Psi: [0, T]\times \O \td \R^Q,\;   \left\VertÊ\Psi  \right\Vert_{H^2(\R^Q)}^2  = \E \int_0^T |\Psi_t|^2 \, dt < \infty.$$
\item[$\bullet$]
$ L^2 (\tl\mu, \R^Q)$ the space of mappings
$ U : \O \times [0, T] \times E \td \R^Q$
which are ${\cal P} \otimes{\mathcal E}$-measurable s.t.
$$ \left\VertÊU   \right\Vert^2_{L^2(\R^Q)}  = \E  \int_0^T \Vert U_t\Vert^2_{L^2(E, {\cal E}, \l, \R)} dt 
< \infty, 
$$
\end{description}
where  ${\mathcal P}$ denotes   the  $\s-$algebra of  ${\cal F}_t-$predictable sets  of $\O\times [0, T]$ and  $$\Vert U_t\Vert^2_{L^2(E, {\cal E}, \l, \R)}  =\int_E \, |U_t (e)|^2\,\l(de).$$
We   may  often   write $|\cdot|$  instead   of   $\Vert\cdot\Vert_{L^2(E, {\cal E}, \l)}$    for  a  sake  of   simplicity.

Notice that the  space ${\cal  B}^2 (\R^Q) = S^2 (\R^Q) \times H^2 (\R^Q)\times L^2 (\tl\mu, \R^Q)$  endowed  with  the  norm  
$$  \left\VertÊ(Y, Z, U) \right\Vert^2_{{\cal  B}^2(\R^Q)}  =  \left\VertÊY  \right\Vert_{S^2(\R^Q)}^2  +  \left\VertÊZ \right\Vert_{H^2(\R^Q)}^2 +  \left\VertÊU \right\Vert_{L^2(\R^Q)}^2 $$
is   a Banach space.

Finally  let  $\bf S$   be  the  set  of   all   non-decreasing continuous  function $\f(\cdot)   : \R_+ \to  \R_+ $  satisfying   $\f(0)=0$  and  $\f(s)>0$  for  $s>0$.  

Let    $f : \O\times [0, T] \times  \R \times  \R^{d} \times L^2(E, {\cal E}, \lambda, \R) \to  \R $ be   jointly   measurable.  
Given $\xi$ a  ${\cal F}-$measurable $\R-$valued random  variable,   we   are  interested  in  the BSDEP with  parameters   $(\xi,   f, T)$:
\begin{equation}
\label{backw}
Y_t= \xi  + \int_t^T f\left(r, \Theta_r\right) dr  - \int_t^TZ_r dW_r  -  \int_t^T\int_E U_r(e) \tl \mu(dr, de),   \; 0\le t  \le T,
\end{equation}
where  $\Theta_r$  stands   for  the   triple  $(Y_r, Z_r, U_r)$. \\
For   instance  let  us   precise  the  notion  of solution   to  \eqref{backw}. 
\begin{definition}
A  triplet  of   processes $(Y_t, Z_t, U_t)_{0\le t \le T}   $ is called  a solution   to  eq.  \eqref{backw},   if $(Y_t, Z_t, U_t)\in  {\cal  B}^2 (\R) 
$ and   satisfies eq. \eqref{backw}.
\end{definition}
First     we  state    some   results   in the   case   of   Lipschitz type  conditions  of  the  generator.  
Suppose  that   assumption ${\bf (A)}$  holds   (where  $0<T\le \infty$) :  

\noindent{\bf (A1)} :  For all $(y, z,  u) \in  \R \times  \R^{d} \times \R,  \; f(\cdot, y, z, u)$  is a  progressively measurable  process   and satisfies
$\E\left[\left(\int_0^T|f(r, 0, 0, 0)| dr\right)^2\right] <\infty$.\\
{\bf (A2)} :   There   exist  two     non-random  functions      $\g (\cdot),   \; \r(\cdot),   \;   : [0, T]\to  \R^+$    
such that     for   $0\!\le \!t \!\le T$  and  $(y, y^\prime) \in \R^2, \; (z, z^\prime) \in (\R^d)^2$ and   $u\in  L^2(E, {\cal E}, \l, \R)$,  
\begin{equation*}
\vert f(t, y, z, u) -  f(t, y^\prime, z^\prime, u)  \vert \le   \g(t) |y - y^\prime| +   \r(t)  |z-z^\prime|. 
\end{equation*}
{\bf(A3)} :  There   exists  $-1<c \le 0$  and   $C>0$, a  deterministic   function $\s(\cdot) : [0, T] \to  \R_+$   and $\b :  \O\times[0, T]\times E \to  \R$,   ${\cal  P}\otimes {\cal E}-$measurable  satisfying  $c(1\wedge |e|) \le  \b_t(e) \le   C(1\wedge |e|)$
such that   for all  $y\in \R,   \; z\in \R^d$ and    $u, u^\prime  \in (L^2(E, {\cal E}, \l, \R))^2$, 
\begin{equation}
 f(t,  y,  z, u)  - f(t,  y,  z, u^\prime)  \le    \s(t) \int_E \left(u(e)   - u^\prime(e)\right)\b_t(e) \l(de). 
\end{equation}   
{\bf (A4)} : The   integrability condition   holds : $\disp   \int_0^T  (\g(s) +  \r^2(s)+\s^2(s))  ds   < \infty.$\\
\begin{rem}
Let   us  mention  that {\bf(A3)}  implies  that  $f$  is   $\s(t)-$Lipschitz  in $u$  since    we   have  (where  $\tl c$ is   a universal   positive   constant)
\begin{align*}
|f(t,  y,  z, u)  - f(t,  y,  z, u^\prime)|  &\le    \tl c\,  \s(t) \int_E \left|u(e)   - u^\prime(e)\right| (1\wedge |e|) \l(de)\\ &\le \tl c \, \s(t)\left(\int_E \left|u(e)   - u^\prime(e)\right|^2 \l(de)\right)^{1/2}
\\&:= \tl c \, \s(t) \Vert u - u^\prime\Vert_{L^2(E, {\cal E}, \l, \R)}
\end{align*}
\end{rem}
We  have  the following   result  which  is  a  consequence   of   Lemma 2.2  in  \cite{Yin-Situ}.
\begin{lemma}\label{exist-Yn}
 Let $\xi \in   L^2(\O,  {\cal F}, \Prb)$ and   $0<T\le \infty$.  If   {\bf (A)}   holds  then    eq. \eqref{backw}   with   parameters  $(\xi,  f, T)$ has   a   unique          solution
$ (Y_t, Z_t, U_t)_{0\le t \le T}$.
\end{lemma}
The  proof   of   our  main  result   need a   comparison  theorem  in infinite  time  horizon.
Given   two  parameters   $(\xi^1, f^1, T)$  and    $(\xi^2, f^2, T)$,  we   consider  the BSDEPs,      $i =1, 2$, 
\begin{equation}
\label{backw-i}
Y^{i}_t= \xi^{i}  + \int_t^T f^{i}\left(r, \Theta^{i}_r\right) dr  - \int_t^TZ^{i}_r dW_r  -  \int_t^T\int_E U^{i}_r(e) \tl \mu(dr, de),   \; 0\le t  \le T,
\end{equation}
where for   $i=1, 2,   \;  \Theta^{i}_\cdot$  stands  for  the   triple   $(Y^{i}_\cdot, Z^{i}_\cdot, U^{i}_\cdot)$. \\
Assume   in addition that   \\
{\bf (A5)} :   $\xi^1\le   \xi^2$  and   $\forall  (\o, t,  y,  z,  u),  \quad     f^1(\o, t, y, z, u) \le f^2(\o, t, y, z, u).$ \\ 

We  have   the   following   result    which  is  proved in \cite{Royer} in  the   case $T<+\infty$ (see Theorem  2.5). The   proof  when  $T=+\infty$ is   given  in   Section \ref{proofs}.
\begin{theorem}\label{comp-Lipsch}
Suppose   that    $f^1$   and    $f^2$  satisfy {\bf (A1)-(A5)} and  $0<T\le +\infty$.  
If  $(Y^{i}_r, Z^{i}_r, U^{i}_r), \; i=1, 2$  are   solutions    to  \eqref{backw-i},  then    we  have  
$$ Y^1_t   \le   Y^2_t,   \quad    \P-a.s.$$
\end{theorem}
\begin{rem}
Theorem \ref{comp-Lipsch}  established  a  comparison  theorem in  the  case   of  Lipschitz coefficients   for either $T<\infty$  or  $T=+\infty$.  Basically  it  improves   the  well   known  result   in  the   finite  time  horizon.   
\end{rem}
Let  us now  deal  with    our   problem.
\subsection{Existence  of   a  minimal solution}
In this section, we will prove   existence of a minimal solution for BSDEPs when their
generators   are continuous and have a linear growth (see Theorem \ref{minimal}  below). First  let  us  give  the  
\begin{definition}
A solution $(Y_t, Z_t, U_t)_{0\le t \le T}   $  of   eq.   \eqref{backw} is called  a minimal   solution     if   for  any    other   solution   $(\tl  Y_t, \tl Z_t, \tl U_t)_{0\le t \le T}$  to   \eqref{backw}  we  have for   each $ 0\le t \le T,   \quad    Y_t \le \tl Y_t.$
\end{definition}

We introduce the following  list  of   conditions  weaker  than  those required   in   \cite{Bar-Buc-Par, Royer, Yao, Yin-Situ}. 

We  assume  that  $0\le T\le +\infty$   and the  generator   $f$  satisfies    assumptions {\bf (H1)} : \\
{\bf (H1.1)} :  There   exist   three  functions    $\g (\cdot),   \, \r(\cdot),   \,  \s(\cdot)  : [0, T]\to  \R^+$ satisfying {\bf (A4)}.\\
{\bf (H1.2)} :  There   exists   a      ${\cal F}_t-$progressively  mesurable  nonnegative process $(f_t)_{0\le t\le T}$ s.t.   $\E\left[\left(\int_0^T f_t dt\right)^2\right] <\infty$ and  for    
$(t, y, z, u)\in [0, T]\times \R\times \R^d\times L^2(E, {\cal E}, \l, \R)$,
$$ \vert f(t,\o,  y, z, u)\vert \le  f_t(\o)  +  \g(t) |y| +   \r(t)  |z| + \s(t)|u|.  $$
{\bf (H1.3)} : $f(\o,  t, \cdot, \cdot, \cdot) :    \R \times  \R^{d} \times L^2(E, {\cal E}, \l, \R)  \to  \R$   is   continuous. \\ 

As  in  \cite{Lep-San},  we   are  led to  consider  the   sequence  $f_n:  \O\times \R \times  \R^{d} \times L^2(E, {\cal E}, \l, \R)  \to  \R$  associated  to   $f$ defined   by $\forall  (t, \omega , y, z, u) \in \O\times \R \times  \R^{d} \times L^2(E, {\cal E}, \l, \R)$  
$$  f_n(t, \omega , y, z, u)  =  \inf_{(y^\prime, z^\prime, u^\prime) \in \R^{1+d} \times L^2(E, {\cal E}, \l, \R)} [f(t, \omega, y^\prime, z^\prime, u^\prime)+
 n (|y-y^\prime|+ |z-z^\prime|+|u-u^\prime|)].$$
Using similar  computations as  in  proof of Lemma 1 in \cite{Lep-San}, one can obtain the following proposition. We
omit its proof. 
\begin{proposition}\label{prelim}
Assume  that  $f$ satisfies   {\bf (H1)}.  Then the sequence of functions $f_n$ is well defined for each $n\ge 1$, and it satisfies, $d{\bf P}\times  dt-$a.s.
\begin{itemize}
\item[(i)] Linear growth: $\forall n\ge 1$,   $\forall   y, z, u, \;  |f_n(\o, t, y, z, u)| \le  f_t (\o) + \g(t)|y| + \r(t) |z|+ \s(t) |u|$.
\item[(ii)] Monotonicity in $n$: $\forall   y, z, u, \quad   f_n(\o, t, y, z, u)$  increases in $n$.
\item[(iii)] Convergence:    $\forall  (\o, t, y, z, u) :   \O\times [0, T]\times \R \times  \R^{d} \times L^2(E, {\cal E}, \l, \R)$,
\begin{equation} \label{convergence} f_n(\o, t, y, z, u) \xrightarrow{n\to +\infty}  f(\o, t, y, z, u).\end{equation}
\item[(iv)] Lipschitz condition:  $\forall n\ge 1$, $\forall  y, y^\prime, z, z^\prime, u, u^\prime$, we have
$$|f_n(\o, t, y, z, u) -  f_n(\o, t, y^\prime, z^\prime, u^\prime)|  \le n\g(t)|y - y^\prime| + n\r(t) |z - z^\prime|  + n \s(t)|u - u^\prime|.$$
\end{itemize}
\end{proposition}
Thus  by  Lemma  \ref{exist-Yn},    the BSDEP   with  parameters  $(\xi,  f_n, T)$: 
\begin{equation}
\label{eq-Yn}
Y_t^n= \xi  + \int_t^T f_n\left(r, \Theta_r^n\right) dr  - \int_t^TZ^n_r dW_r  -  \int_t^T\int_E U^n_r(e) \tl \mu(dr, de),   \; 0\le t  \le T,
\end{equation}
has   a  unique solution   $(\Theta^n_t)_{0\le t \le T} = (Y^n_r, Z^n_r, U^n_r)_{0\le t \le T}$.

The   Main    result  in  this  section  is  the   following 
\begin{theorem}\label{minimal}
Let $\xi \in   L^2(\O,  {\cal F}_T, \Prb)$ and   $0<T\le \infty$. Under  assumption  {\bf (H1)},   the   BSDEP \eqref{backw}  has   a     minimal   solution
$ (Y_t, Z_t, U_t)_{0\le t \le T}$.
\end{theorem}
\begin{proof}
We  follow  the  proof  of Theorem  1 in   \cite{She-Lon-Tian}. Consider  $F:  \O\times [0, T]\times \R \times  \R^{d} \times L^2(E, {\cal E}, \l, \R)\to \R$ given   by  
$$\forall (\o, t, y, z, u),  \quad    F(\o, t, y, z, u)  =   f_t (\o) + \g(t)|y| + \r(t)|z|+  \s(t)|u|.$$ 
It   follows from     Lemma   \ref{exist-Yn}  that  the   BSDEP   with   parameters  $(\xi, F, T)$  admits   a   unique solution   $(\tl Y_t, \tl Z_t, \tl U_t)_{0\le t \le T}$. Applying Theorem \ref{comp-Lipsch} and Proposition \ref{prelim}, we deduce   that       
$$ \forall  (\o, t) \in \O\times [0, T],   \quad   Y_t^1(\o)\le Y_t^n(\o)\le Y_t^{n+1}(\o)\le   \tl Y_t(\o).    $$
Hence     there    exists   a  ${\cal F}_t-$progressively  measurable  process   $(Y_t)_{0\le t \le T}$  such  that  $\disp   \lim_{n\to + \infty} Y_t^n(\o)  =  Y_t(\o)$. Putting     
$G =  \disp  \sup_n \sup_{0\le s \le T} |Y_s^n(\o)|$,   arguing  as  in   \cite[Theorem 1]{She-Lon-Tian}   we   have      
$$  \E\left( \sup_{0\le s \le T} |Y_s(\o)|^2\right)  \le   \E(G^2) < \infty.$$
It\^o's  formula applied  to   eq. \eqref{eq-Yn},  yields  ($0\le t \le T$)
\begin{align*}
\mathbf{E} |Y_t^n|^2  + \mathbf{E}\int_t^T \!\!|Z_r^n|^2  dr  &+ \mathbf{E}\int_t^T\!\!\int_E\!\! |U_r^n(e)|^2 \lambda(de) dr   =  \mathbf{E}|\xi|^2  
+  2\mathbf{E}\int_t^T \!\!Y_r^n \, f_n\left(r, \Theta_r^n\right) dr  \\
&\le   \mathbf{E}|\xi|^2  
+  2\mathbf{E}\int_t^T \!\!|Y_r^n| \left(f_r + \g(r) |Y_r^n |+  \r(r) |Z_r^n | + \s(r) |U_r^n |\right) dr 
\end{align*}
Using the  inequality  $2ab  \le   a^2\eps +  (b^2/\eps)$  for   every  $a\ge0, b\ge 0$  and  $\eps>0$, we  deduce  that       (where  $\d    =  2\int_0^T  \r^2(s)  ds$  and $\d^\prime    = 2 \int_0^T  \s^2(s)  ds$)
\begin{align*}
 \mathbf{E}\int_0^T \!\!|Z_r^n|^2  dr  &+ \mathbf{E}\int_0^T\!\!\int_E\!\! |U_r^n(e)|^2 \lambda(de) dr   \le   \mathbf{E}|\xi|^2  
+  (1 + \d + \d^\prime) \mathbf{E}(G^2) \\
&+  \E\left[\left(\int_0^T   f_r  dr\right)^2\right] +  2 \E(G^2)\cdot\int_0^T \g(r)  dr \\
&+  \frac{1}{\d} \E\left[\left(\int_0^T   \r(r) |Z_r^n|  dr\right)^2\right] 
+ \frac{1}{\d^\prime} \E\left[\left(\int_0^T  \int_E\!\!   \s(r) |U_r^n(e)| \lambda(de)  dr\right)^2\right]. 
\end{align*}
Applying  H\"older's  inequality  in  the    two last   integrals,    we  obtain   
$$  \mathbf{E}\left[\int_0^T \!\!|Z_r^n|^2  dr  + \int_0^T\!\!\int_E\!\! |U_r^n(e)|^2 \lambda(de) dr\right]   \le  M + \, \frac{1}{2} \mathbf{E}\left[\int_0^T \! |Z_r^n|^2  dr  + \int_0^T \!\int_E\!\! |U_r^n(e)|^2 \lambda(de) dr\right] $$   
where    $$  M = \mathbf{E}|\xi|^2  
+  (1 + \d + \d^\prime) \mathbf{E}(G^2) 
+  \E\left[\left(\int_0^T   f_r  dr\right)^2\right] +  2 \E(G^2)\cdot\int_0^T \g(r)  dr>0$$   and    depend  only  on the  parameters   $f, \xi$  and  $T$.
Consequently    we   have  
$$\sup_{n\in\N}\mathbf{E}\int_0^T \!\!|Z_r^n|^2  dr    \le  2M  \quad   \mbox{and}\quad \sup_{n\in\N} \mathbf{E}\int_0^T\!\!\int_E\!\! |U_r^n(e)|^2 \lambda(de) dr \le  2M.$$
Let  us  define for   ${\cal W}  \in \{Y,   Z, U\}$  and  integers  $n,  m\ge 1,   \;   {\cal W} ^{n, m}  =  {\cal W} ^n - {\cal W} ^m$.
 
\noindent Applying  again  It\^o's formula, we  deduce   from  \eqref{eq-Yn}, 
\begin{align*}
\mathbf{E} |Y_t^{n, m} |^2  + \mathbf{E}\int_t^T |Z_r^{n, m}|^2  dr  &+ \mathbf{E}\int_t^T\int_E |U_r^{n, m}(e)|^2 \lambda(de) dr\\
   &=   2  \mathbf{E}\int_t^T Y_r^{n, m}\,  (f_n\left(r, \Theta_r^n\right) - f_m\left(r, \Theta_r^m\right) )dr, \quad   0\le t \le T.  
  \end{align*}  

Using  once   again    H\"older's  inequality   and   assumption  {\bf (H1)}  we  obtain   
\begin{align*}
\mathbf{E} |Y_0^{n, m} |^2  &+ \mathbf{E}\int_0^T |Z_r^{n, m}|^2  dr  + \mathbf{E}\int_0^T\int_E |U_r^{n, m}(e)|^2 \lambda(de) dr\\
   &\le  4  \mathbf{E}\int_0^T |Y_r^{n, m}| f_r  dr  + 4\left(\E(G^2)\right)^{1/2}\cdot \left(\E\left[\left(\int_0^T |Y_r^{n, m}| \g(r)  dr  \right)^2\right]\right)^{1/2} \\
   &+ 2 \sqrt{8 M} \cdot \left(\E\left[\int_0^T |Y_r^{n, m}|^2 \r^2(r)  dr \right]\right)^{1/2}  +  2 \sqrt{8 M} \cdot \left(\E\left[\int_0^T |Y_r^{n, m}|^2 \s^2(r)  dr \right]\right)^{1/2}    
  \end{align*}  
In particular      Lebesgue's   dominated  convergence     theorem  implies    that $\{Z^n\}$ (respectively  $\{U^n\}$)    is   a   Cauchy  sequence   in   $H^2(\R^d)$  (respectively  $L^2(\tl  \mu, \R)$). Hence  there   exists    $(Z, U)  \in H^2(\R^d)\times L^2(\tl  \mu, \R)$   such  that  
$$  \left\Vert Z^{n}  -  Z \right\Vert_{H^2(\R^d)}^2 \to  0   \quad   \mbox{and}    \quad   \left\Vert  U^{n}-  U \right\Vert_{L^2(\R)}^2 \to 0,  \quad \mbox{as}    \quad  n \to \infty, $$  
which  implies along   a  subsequence   if   necessary
$$ Z^n    \xrightarrow{H^2(\R^d)}   Z    \quad   \mbox{and}    \quad   U^n \xrightarrow{L^2(\R, \tl\mu)} U,   \quad    \mbox{as} \;   n \to  \infty.$$
Further  by   virtue   of   \eqref{convergence},   we  have 
$ f_n(s, Y^n_s, Z^n_s, U^n_s)     \xrightarrow{n\to  \infty}   f(s,   Y_s, Z_s, U_s),  \; \; 0\le s \le T$
 and   arguing    as   in   \cite[Theorem 1]{She-Lon-Tian},   we   deduce  that  
$$\lim_{n\to \infty}\E\left[\left(\int_0^T |f_n(r, \Theta^n_r)  -  f(r, \Theta_r)| dr\right)^2\right]  = 0 \quad
\mbox{and} 
 \quad\lim_{n\to \infty}\E\left(\sup_{0\le t \le T}|Y^n_t-  Y_t|^2\right)= 0.$$  This  is  enough  to  deduce     that     $Y\in S^2(\R)$. Letting  $n\to + \infty$  in   \eqref{eq-Yn}, we   prove  that  
$(Y_s, Z_s, U_s)_{0\le s\le T}$  is   solution    to  \eqref{backw}.   

Let   $(Y^\prime, Z^\prime, U^\prime) \in {\cal  B}^2(\R) $   be   a  solution   of   eq.  \eqref{backw}. Thanks   to     Theorem  \ref{comp-Lipsch},   we  have 
$$\forall  n \ge 1,   \quad   Y^n \le Y^\prime.$$ Letting     $n \to \infty$,  we  get  $Y\le Y^\prime.$    This    implies  that  $Y$ is  the   minimal  solution to  \eqref{backw}.  
\end{proof}
\section{Comparison theorem}
We  intend  to  prove   a   comparison  theorem  under  mild   conditions  on  the  drift  of  the  BSDEP. This     result  is  useful  for  the   proof  of   existence  and  uniqueness   of  solution.    

Let  us  introduce  the  following   assumptions {\bf(H2)} on  the   generator  $f$ where      $0<T\le +\infty$.\\ 
\noindent{\bf(H2.1)}: $f $ is   weakly  monotonic  in   $y$ i.e.  there   exists  $\g(\cdot)  :  [0, T] \to  \R_+$  satisfying    $\int_0^T  \g(t) dt < \infty  $  and    a  function  $\varrho  \in {\mathbf S}$ s.t.  $\int_{0^+}\frac{1}{\varrho(r)} dr = + \infty$  and  for   any   $(y, y^\prime)  \in \R^2,   \;  z\in \R^d,   \; u\in L^2(E, {\cal E}, \lambda, \R)$,  
\begin{equation} (y - y^\prime) \left(f(t,  y,  z, u)  -f(t,  y^\prime,  z, u)\right) \le   |y - y^\prime|  \g(t) \varrho(|y - y^\prime|) 
\end{equation}
and   we  assume  that     $\varrho(x)  \le   k(x+1) $ where    $k$ denotes   the  linear  growth  constant   of  $\varrho$. 
 
\noindent{\bf(H2.2)}:   $f$ is  uniformly continuous   in  $z$ and   there  exists  $\r(\cdot) : [0, T] \to  \R_+$  satisfying   $\int_0^T   \r^2(t)  dt < \infty$  and   $\phi \in  {\mathbf S} $ s.t.
$$  |f(t,  y,  z, u)  -f(t,  y,  z^\prime, u)| \le    \r(t) \phi (|z - z^\prime|) $$ 
and    we assume  that    $\phi(x)  \le   ax+b,   \;  a>0, b>0.$

\noindent{\bf(H2.3)}:  There   exists  $-1<c \le 0$  and   $C>0$, a  deterministic   function $\s(\cdot) : [0, T] \to  \R_+$  satisfying   $\int_0^T   \s^2(s)  ds < \infty$   and $\b :  \O\times[0, T]\times E \to  \R$,   ${\cal  P}\otimes {\cal E}-$measurable  satisfying  $c(1\wedge |e|) \le  \b_t \le   C(1\wedge |e|)$
such that   for all  $y\in \R,   \; z\in \R^d$ and    $u, u^\prime  \in (L^2(E, {\cal E}, \l, \R))^2$, 
\begin{equation}
 f(t,  y,  z, u)  - f(t,  y,  z, u^\prime)  \le    \s(t) \int_E \left(u(e)   - u^\prime(e)\right)\b_t(e) \l(de). 
\end{equation}   
Given two parameters  $(\xi^1, f^1)$  and  $(\xi^2, f^2)$,          we are  interested  in     two   one-dimensional  BSDEPs 
(with  $0\le t  \le T$)
\begin{align}
\label{backw-y1} Y^1_t&= \xi^1  \!\!+ \int_t^T\!\! f^1\left(r, \Theta^1_r\right) dr   - \int_t^T\!\!Z^1_r dW_r -  \int_t^T\!\!\int_E \!\!U^1_r(e) \tl \mu(dr, de),  \\
Y^2_t&= \xi^2  + \!\! \int_t^T\!\! f^2\left(r, \Theta^2_r\right) dr   - \int_t^T \!\!Z^2_r dW_r  - \!\! \int_t^T\!\!\int_E U^2_r(e) \tl \mu(dr, de),\label{backw-y2}
\end{align}
and    we   
assume    in addition  that

\noindent {\bf (H2.4)}: $\forall\,(t, y, z,  u),  \quad   f^1(t, y, z, u) \le  f^2(t, y, z, u) $  and  $\xi^1  \le  \xi^2$.  
  
We state  the   following  result   (see   \cite[Lemma 3]{She-Lon-Tian})    which  will  be   useful in  the  sequel
\begin{lemma}\label{Phi}
Let  $\Psi(\cdot) :  \R^+  \to   \R^+ $  be   a   nondecreasing   function  with linear  growth  which   means  
$$\exists   K>0  \; \;   \mbox{s.t.   for all}\;   \;   x\in \R^+,  \quad   \Psi(x)  \le  K(x+1).$$
Then   for   each  $n\ge 2 K$  we  have,   $\displaystyle \Psi\left(x\right)   \le   nx   +  \Psi\left(\frac{2K}{n}\right), \quad   x\ge 0.$
\end{lemma}
Before  proving  the  main  statement  of  this  section,    let  us  recall   the  Girsanov  theorem   for  discountinuous   processes.    
If   ${\cal M}^2$  denotes the set  of   square  integrable  martingales,  we  can   define   thanks  to  the   martingale  representation  (see  \cite[Lemma 2.3]{Tang-Li}) a  mapping  
%
%
\begin{align*}
\Phi :  {\cal M}^2 &\rightarrow   H^2(\R^d) \times L^2(\tl \mu, \R)  \\
M&\mapsto  (\theta,  \upsilon)\quad \mbox{such  that}\quad    M_t  =  \int_0^t   \theta_s d W_s   +   \int_0^t\int_E  \upsilon_r(e)  \tl \mu(de, dr).
\end{align*} Let  
$
{\cal M} =  \bigg\{(M_t)_{t\ge 0} \in {\cal M}^2 \big\vert ||\th_sÊ||\le C,   \; \; \upsilon_s(x)>-1, \; \,  \;            |\upsilon_s(x)|\le C(1\wedge |x|), a.s.   \;   \mbox{with}\; \;  \Phi(M)=(\theta, \upsilon)\bigg\}
$.
For   $M \in {\cal M}$,   the  Dol\'eans-Dade exponential  of   $M$ is   defined   by   
$${\cal  E}  (M)_T =  e^{M_T - \frac{1}{2}\langle M^c\rangle_T} \prod_{0<s\le T}(1+ \Delta M_s)e^{- \Delta M_s}.$$ We  have 
\begin{theorem}[Girsanov Theorem]\label{girsanov}
Let  $(\bar Z, \bar U) \in  H^2(\R^d) \times L^2(\tl \mu, \R)$ and   $K_t =   \int_0^t   \bar Z_s d W_s   +   \int_0^t\int_E  \bar U_r(e)  \tl \mu(de, dr)$. If   $M \in {\cal  M}$  then    the  process  
$\tl K   =  K  -  \langle K, M \rangle $  is    a  martingale  under  the  probability  measure  $\tl  {\bf P}$  s.t $  d {\tl {\bf P}}/ d {\bf P} =  {\cal  E}(M)_T$.
\end{theorem}
Here  is  the   main  result   of  this  section. 
\begin{theorem} \label{compare}Let   $0<T\le+\infty$. Assume   given     $f^1$,    $f^2$ and  $(\xi^1, \xi^2)  \in   (L^2(\O,  {\cal F}_T, \Prb))^2$ such that  {\bf (H2)}  holds. If   
$(Y^1_t, Z^1_t, U^1_t)_{0\le t \le T}$  and $(Y^2_t, Z^2_t, U^2_t)_{0\le t \le T}$  are   solutions   of eq.  \eqref{backw-y1}  and eq.  \eqref{backw-y2}  respectively,  then we   have
$$ \forall   \,  0\le t \le T,  \qquad  Y^1_t \le Y^2_t,   \quad   \P-a.s.$$ 
\end{theorem}
\begin{proof} We  assume  $d =1$. Putting  
\begin{equation}\label{notation} \widehat{\Theta}_t  = (\widehat{Y}_t, \widehat{Z}_t, \widehat{U}_t)  = (Y^1_t-Y^2_t, Z^1_t-Z^2_t , U^1_t-U^2_t),  \quad \widehat{\xi}  =  \xi^1-\xi^2,\end{equation}  then  $(\widehat{\Theta}_t ) _{0\le t \le T}$ satisfies   the  BSDEP ($ 0\le t \le T$)
\begin{equation}
\label{backw-Theta}
\widehat{Y}_t =\widehat \xi  + \int_t^T \left[f^1\left(r, \Theta^1_r\right) -f^2\left(r, \Theta^2_r\right) \right] dr  - \int_t^T\widehat{Z}_r dW_r  -  \int_t^T\int_E \widehat{U}_r(e) \tl \mu(dr, de).
\end{equation}
 Tanaka-Meyer's  formula yields (where  $x^+=\max(x, 0)$)
\begin{align}
\label{Yt+}
\widehat{Y}_t^+ &\le  \widehat{\xi}^+ +  \int_t^T   \mathds{1}_{\{\widehat{Y}_r^+>0\}}\left[f^1\left(r, \Theta^1_r\right) -f^2\left(r, \Theta^2_r\right) \right] dr 
-\int_t^T \mathds{1}_{\{\widehat{Y}_r^+>0\}}  \widehat{Z}_r dW_r  \nonumber \\
&- \int_t^T \int_E  \mathds{1}_{\{\widehat{Y}_r^+>0\}} \widehat{U}_r(e) \widetilde \mu(de, dr), \quad   0\le t \le T.   
\end{align}
Further      we   have   
$$
f^1\left(r, \Theta^1_r\right) -f^2\left(r, \Theta^2_r\right) =  [f^1\left(r, \Theta^1_r\right) -f^1\left(r, \Theta^2_r\right)]  + [f^1\left(r, \Theta^2_r\right) -f^2\left(r, \Theta^2_r\right)]
$$
and      assumption   {\bf (H2.4)}  implies  that  the   right-hand  side   is   less  than   
\begin{align*}
[f^1\left(r, \Theta^1_r\right) -f^1\left(r, Y^2_r, Z^1_r, U^1_r\right)] &+  [f^1\left(r, Y^2_r, Z^1_r, U^1_r\right)  - f^1\left(r, Y^2_r, Z^2_r, U^1_r\right) ]   \\
&+ [f^1\left(r, Y^2_r, Z^2_r, U^1_r\right) - f^1\left(r, \Theta^2_r\right)].   
\end{align*}
Hence   applying  {\bf (H2.1)}   and    {\bf (H2.3)}   we  deduce    that  
\begin{align*}
\mathds{1}_{\{\widehat{Y}_r^+>0\}}\left[f^1\left(r, \Theta^1_r\right) -f^2\left(r, \Theta^2_r\right) \right]  &\le  
\g(r)  \varrho(\widehat{Y}_r^+)   +\mathds{1}_{\{\widehat{Y}_r^+>0\}}   \r(r)  \phi( |\widehat{Z}_r|)\\
&+   \int_E \mathds{1}_{\{\widehat{Y}_r^+>0\}}  \widehat{U}_r(e)\b_r(e) \l(de).
\end{align*}
By   Lemma \ref{Phi}    we  have  (with    $\Psi(\cdot) =  \phi(\cdot);   \, K=c= a+b$) 
$$ \mathds{1}_{\{\widehat{Y}_r^+>0\}}   \r(r)  \phi( |\widehat{Z}_r|) \le \mathds{1}_{\{\widehat{Y}_r^+>0\}}   n \r(r) |\widehat{Z}_r|
+ \mathds{1}_{\{\widehat{Y}_r^+>0\}}    \r(r)    \phi\left(\frac{2c}{n}\right),  \quad   n\ge 2c.$$
Putting  pieces   together,   we  derive  from  \eqref{Yt+}
\begin{equation}
\label{Yt+1}
\widehat{Y}_t^+ \le a_n +  \int_t^T  \g(r)  \varrho(\widehat{Y}_r^+)  dr +   \tl  K_t
\end{equation}  
where
\begin{align*} 
 \tl  K_t=\int_t^T\bigg[\mathds{1}_{\{\widehat{Y}_r^+>0\}} \widehat{Z}_r \left(\frac{n \r(r) \widehat{Z}_r}{|\widehat{Z}_r|}\mathds{1}_{\{\widehat{Z}_r\neq 0\}}\right)
    &+   \int_E\mathds{1}_{\{\widehat{Y}_r^+>0\}}    \widehat{U}_r(e)  \b_r(e)  \l(de) \bigg] dr  \\ -\int_t^T \mathds{1}_{\{\widehat{Y}_r^+>0\}}  \widehat{Z}_r   d W_r 
    &- \int_t^T \int_E  \mathds{1}_{\{\widehat{Y}_r^+>0\}} \widehat{U}_r(e) \widetilde \mu(de, dr) 
\end{align*}
and  (where  $b$  is  given  in  (H2.2))    
$$  a_n  =   \mathds{1}_{b\neq 0} \phi\left(\frac{2c}{n}\right)\cdot   \int_0^T  \g(r)  dr  \;   \xrightarrow{n\to\infty} 0.
$$
Define    
\begin{align*}
M_t  &=  \int_0^t  \left(\frac{n \r(r) \widehat{Z}_r}{|\widehat{Z}_r|}\mathds{1}_{\{\widehat{Z}_r\neq 0\}}\right)  d W_r   +  \int_0^t \int_E \b_r(e)  \widetilde \mu(de, dr), \quad   0\le t\le T, \\
K_t&=  \int_0^t \mathds{1}_{\{\widehat{Y}_r^+>0\}}  \widehat{Z}_r   d W_r 
    + \int_0^t \int_E  \mathds{1}_{\{\widehat{Y}_r^+>0\}} \widehat{U}_r(e) \widetilde \mu(de, dr),  \quad   0\le t\le T.
 \end{align*} 
By  Theorem \ref{girsanov}, it follows  that  $\tl K_t$ is   a  martingale  under  the  probability  measure ${\bf \tl P} =  {\cal E}(M)_T \cdot {\bf P}$.     Hence   taking $\tl\E\left(\cdot\vert {\cal F}_t\right) $ the   conditional  expectation   given  ${\cal F}_t$   under   the  probability  measure  $\tl{\bf P}$, and   taking  in account   $\varrho $ is  concave, we deduce  that
$$\tl \E \left(\widehat{Y}_s^+\vert {\cal F}_t\right)  \le  a_n  +  \int_s^T \g(r) \varrho\left(\tl\E \left[\widehat{Y}_r^+\vert {\cal F}_t\right]\right) dr,   \quad    t\le s \le T.$$
Thus     Lemma 5  in   \cite{She-Lon-Tian}   implies that   $  \widehat{Y}_t^+=0$    
which  is  true  if  and  only  if $Y_t^1 \le  Y_t^2.$
\end{proof}
The    following corollary  is   immediate.  
\begin{corollary}\label{cor-existence}
Let  $0<T\le +\infty$. If   $\xi\in L^2(\O, {\cal  F}, {\bf P})$   and  $f$  satisfies   {\bf (H2)},   then    the  BSDEP  \eqref{backw}  with   parameters  $(\xi, f, T)$  has  at  most   one  solution.
\end{corollary}
\section{Existence   and   uniqueness of solution}
Thanks  to  the  results  establish  in the  previous  section, we  investigate  in  this  section    the solvabilty  of  our   equations under   weaker conditions  on  the  generator.   
 
Assume  that   $f : \O\times [0, T] \times  \R \times  \R^{d} \times L^2(E, {\cal E}, \lambda, \R) \to  \R $  is  uniformly    continuous  with  respect to  its   variables    and   satisfies  {\bf (H3)} :   
\begin{align*} 
|f(t,  y,  z, u)  -f(t,  y^\prime,  z^\prime, u)| &\le  \g(t) \varrho(|y - y^\prime|) +  \r(t) \phi (|z - z^\prime|), \\
f(t,  y,  z, u)  -f(t,  y,  z, u^\prime)&\le \s(t) \int_E \left(u(e)   - u^\prime(e)\right)\b_t(e) \l(de) 
\end{align*} 
 where    $\g,  \r, \s, \phi $  and   $\b$  are    as  in  {\bf (H2)}.
 
 We  claim 
 \begin{theorem}
 Let  $0<T\le +\infty$   and  $\xi  \in L^2(\O, {\cal F},\P )$. If $f$  satisfies     {\bf (H3)}  and    {\bf (A1)}  then    equation   \eqref{backw}  admits  a   unique solution. 
 \end{theorem} 
 \begin{proof}
 Uniqueness   follows  from  Corollary  \ref{cor-existence}  since    {\bf (H3)}  implies     {\bf (H2)}.  Moreover from     {\bf (H3)}    one  can  derive  that
 \begin{align*}
|f(\o, t,  y,  z, u)| &\le  \g(t) \varrho(|y|) +  \r(t) \phi (|z|) +\tl  c\,  \s(t) \left(\int_E \left |u(e)\right|^2 \l(de)\right)^{1/2}  +    |f(\o, t,  0, 0, 0)|\\
&\le f_t +   k\g(t)|y|  +  a \rho(t)|z|+  \tl c \,  \s(t) |u| 
 \end{align*}
 where    $f_t  =   k\g(t)+ b \rho(t) +   |f(\o, t,  0, 0, 0)|$. Hence   Theorem  \ref{minimal}  ensures  existence    of   a  minimal solution.  This   completes   the  proof. 
 \end{proof}
\section{Proof  of Theorem \ref{comp-Lipsch}}\label{proofs}

This  section  is  devoted to  establishing the comparison theorem under   assumptions {\bf (A1)-(A5)} and  a  horizon   time   $T$  satisfying $0<T\le +\infty$. We   consider   the case  $T=+\infty$  since  the   result  for  $T<\infty$ is   well   known.    The  key  point  is  to  expressed the  difference  of two  solutions   as   a  conditional  expectation  in  a  suitable  probability   space. To  do  this  we need to apply  Girsanov   theorem. This is the guiding line of the following computations.
To begin with, let us establish the following result.

\begin{proposition}
 \label{eq-lineaire} Let  $(a_t)_{t\ge 0}, (b_t)_{t\ge 0}$   be adapted  processes  satisfying a.s.   $|a_t|\le \g(t);   \; |b_t|\le \r(t)$. 
Assume  that  there   exist  a constant  $ C_{\ref{eq-lineaire}}>0$ and      a  process  $(\a_t)_{t\ge 0}$   satisfying  $\a_t(e) >-1$  and 
 $|\a_t(e)|   \le C_{\ref{eq-lineaire}}(1\wedge |e|) \, a.s.$   and   an adapted   process   $\{\f_t\}_{t\ge 0}$   satisfying  $\E\left[\left(\int_0^\infty|\f_t|dt\right)^2\right] <\infty$. 
If   $(Y_t, Z_t, U_t)$  is   solution to      the  BSDEP   
 \begin{align}
\label{backw-lineaire} 
Y_t&= \xi  \!\!+\int_t^\infty \!\! \left(\f_s +  a_s Y_s +  b_s Z_s  +  \int_E \a_s(e) U_s(e) \l(de) \right) ds \nonumber\\
&- \int_t^\infty\!\!Z_s dW_s -  \int_t^\infty\!\!\int_E \!\!U_s(e) \tl \mu(ds, de),  \quad  t\ge 0.
\end{align}
then  there   exists   a  probability  measure  $\tl  \P$  such  that  
$$ Y_t =  
\tl\E\left[\xi  \exp\left(\int_t^\infty  a_s   ds\right) +  \int_t^\infty \f_s\exp\left(\int_t^s   a_r  dr\right)ds\bigg\vert {\cal F}_t  \right], \quad   t\ge 0, 
$$
where   $\tl\E$ stands   for  the   expectation  under  $\tl\P$.
\end{proposition}

\begin{proof}Thanks   to assumptions  on  $b$  and  $\a$,  it  is   easily  seen   that  the  stochastic process $M=(M_t)_{0\le t\le T}$  given  by  
$$  M_t  = \int_0^t b_s  dW_s +  \int_0^t \int_E  \a_s(e)  \tl \mu(ds, de),   \quad   0\le t \le T$$  belongs  in   ${\cal  M}^2$. So let  ${\cal E}(M)_t$   be  the   Dol\'eans-Dade    exponential  of $M$.   By    Theorem \ref{girsanov}  there   exists  a  probability measure  $\tl\P$  such  that  

\begin{align*}
\frac{d \tl\P}{d\P}\big\vert {\cal F}_t   &=  {\cal E} (M)_t\\
&= \exp\left(\int_0^t b_s  dW_s  -  \frac{1}{2} \int_0^t b_s^2  ds\right) \times \\ 
&\prod_{0\le s \le t}\left(1 +  \int_E \a_r(e) \mu(\{s\}, dr)\right)  \exp\left(-  \int_0^t\int_E \a_r(e) \l(de) dr\right).
\end{align*} 
Moreover  the  process  $W_t^*  = W_t   -  \int_0^t b_r dr  $   is   a   Brownian   Motion  under   $\tl \P$   and   $\mu^*(dr, de)  =  \tl \mu(dr, de) -  \a_r(e) \l(de) dr $ is   $\tl \P-$  martingale.  Let  $0<T<\infty$  be fix. One can see   that   $Y_t$ can  be     rewritten  as  
$$Y_t= Y_T +\int_t^T \!\! \left(\f_r +  a_ r Y_r\right) dr - \int_t^T\!\!Z_r dW^*_r -  \int_t^T\!\!\int_E \!\!U_r(e) \mu^*(dr, de),  \quad  0\le t  \le T<\infty.  $$  
Define   $\G_t(\o) = e^{\int_0^t a_r(\o)  dr}, \; \o\in \O, \; t\ge 0$.  It  follows    from    It\^o's   formula
 \begin{align*}
  \G_{t}  Y_t  &=  \G_{T}  Y_T  -  \int_t^T \! \G_{r^-} d Y_r  -  \int_t^T \! Y_{r^-}  d \G_{r}  - \int_t^T \! d [Y,  \G]_r \\
  %
  %
  &=   \G_{T}  Y_T +  \int_t^T \! \G_{r^-}  \f_r  d r   -   \int_t^T \! \G_{r^-}  Z_r  d W_r^* -  \int_t^T  \int_E \! \G_{r^-}
  U_r(e)\mu^*(dr, de)
   \end{align*}
   Taking conditional  expectation $\tl\E (\cdot| {\cal F}_t)$,    we  deduce  that for   any  $0<t\le T<\infty$, 
   $$ \G_t Y_t  = \tl \E\left[\G_{T} Y_T +  \int_t^T \f_r \G_{r^-} dr\bigg\vert {\cal F}_t  \right] $$
   which  implies   $$ Y_t = \tl\E\left[ Y_T\exp\left(\int_t^T  a_s  ds\right) +  \int_t^T \f_s\exp\left(\int_t^s   a_r dr\right)ds\bigg\vert {\cal F}_t  \right].$$
Letting   $T\to  \infty$,   we  deduce  that  $$ Y_t = \tl\E\left[\xi  \exp\left(\int_t^\infty  a_s  ds\right) +  \int_t^\infty \f_s\exp\left(\int_t^s   a_r  dr\right)ds\bigg\vert {\cal F}_t  \right].$$

 \end{proof}
\begin{lemma}
Assume   given     $f^1, \, f^2$ and  $(\xi^1, \xi^2)  \in   (L^2(\O,  {\cal F}_T, \Prb))^2$ such that   {\bf (A1)-(A4)}  hold. If  $(\Theta^{i}_r) = (Y^{i}_r, Z^{i}_r, U^{i}_r)$  is   the  corresponding   solution,  then  there    exists   a  probability  $\tl  \P$  such  that  
$$ \widehat{Y}_t =  \tl\E\left[ \widehat{\xi}  \exp\left(\int_t^\infty  a_s  ds\right) +  \int_t^\infty \left[f^1(s, \Theta^2_s) - f^2(s, \Theta^2_s)\right]\exp\left(\int_t^s   a_r  dr\right)ds\bigg\vert {\cal F}_t  \right] $$
where    $ \widehat{Y}_t $  and   $ \widehat{\xi}$  are   given   by   \eqref{notation}. 
\end{lemma}
\begin{proof}
W.l .o.g   we  assume  $d =1$. Define   $\f_s   =   f^1\left(s, \Theta^2_s\right) -f^2\left(s, \Theta^2_s\right)$ and  
%
\begin{align*}
\D_y f^1(s)   &=  \frac{f^1(Y^1_s, Z^1_s, U^1_s) - f^1(Y^2_s, Z^1_s, U^1_s)}{\widehat{Y}_s}\mathds{1}_{\{\widehat{Y}_s\neq 0\}},   \\  
 \D_z f^1(s)   &=  \frac{f^1(Y^1_s, Z^1_s, U^1_s) - f^1(Y^1_s, Z^2_s, U^1_s)}{\widehat{Z}_s}\mathds{1}_{\{\widehat{Z}_s\neq 0\}},   \\
  \D_u f^1(s, e)   &=  \frac{f^1(Y^1_s, Z^1_s, U^1_s(e)) - f^1(Y^1_s, Z^1_s, U^2_s(e))}{\widehat{U}_s(e)}\mathds{1}_{\{\widehat{U}_s(e)\neq 0\}}
\end{align*}
Then  $(\widehat{\Theta}_t ) _{0\le t \le T}$ is  solution   to   
\begin{align}
\label{backw-Theta}
\widehat{Y}_t &=\widehat \xi  + \int_t^T \left(\f_s  + \D_y f^1(s)\widehat{Y}_s +\D_zf^1(s)\widehat{Z}_s+ \int_E \D_u f^1(s, e)\widehat{U}_s (e) \l(de)\right)ds 
\nonumber\\
& - \int_t^T\!\!\widehat{Z}_s dW_s  -  \int_t^T\!\!\int_E \widehat{U}_s(e) \tl \mu(ds, de).
\end{align}
By  assumptions  on  the  generator  $f^1$, we  have   
$$ |\D_y f^1(s)|\le  \g(s),     \quad   |\D_zf^1(s)| \le  \r(s),    \quad |\D_u f^1(s,e)| \le C(1\wedge |e|) \;  \mbox{and}  \; \D_u f^1(s, e)>-1$$
Hence   applying Proposition \ref{eq-lineaire} with  $a_s= \D_y f^1(s)$,  $b_s= \D_zf^1(s) $   and   $\a_s(e) =  \D_u f^1(s,e)$ we get  the desired   result. 
\end{proof}

{\bf Proof  of  Theorem \ref{comp-Lipsch}}: 
Applying   the   previous  Lemma   and  taking in  account   assumptions   {\bf (A1)},  we  deduce  that   $ \widehat{Y}_t \le 0$   since $\widehat \xi \le0$ and        $ \f_s  \le 0.$


\begin{thebibliography}{5}
\scriptsize
\bibitem{Bah} Bahlali, K.,   Backward stochastic differential equations with locally Lipschitz coefficient, {C. R. Acad.Sci. Paris, Ser. I}, {\bf  333}, (2001), 481--486.
\bibitem{Bar-Buc-Par} Barles, G.,  Buckdahn, R., Pardoux, \'E.,  Backward stochastic differential equations and integral-partial differential
equations, {Stochastics  and Stochastics Reports}, {\bf 60},  (1996), 57--83.
\bibitem{Che-Wan} Chen, Z.,   Wang, B., Infinite time  interval BSDEs and  the   convergence  of   $g$-martingales, {J. Austral. Math. Soc. (Series A)},  {\bf 69},  (2000) 187--211.
\bibitem{Dar-Par} Daling, R.,    Pardoux, \'E., BSDE with random terminal time and applications to semilinear elliptic PDE, {Ann. Probab.} {\bf 3}, (1997), 1135--1159.
\bibitem{Fan-Jia}  Fan, S.,  Jiang, L.,  Finite and   infinite time  interval BSDEs with  non-Lipschitz coefficients, {Statistics  and Probability  Letters},  {\bf 80}, (2010), 962--968.
 \bibitem{She-Lon-Tian}  Fan, S.,   Jiang, L.,   Tian, D.,  One dimensional  BSDEs  with   finite  and  infinite   time horizon, {Stochastic Processes   and their  Applications},  {\bf 121}, (2011), 427--440.
 
\bibitem{Kob} Kobylanski, M.,   R\'esultats   d'existence et  d'unicit\'e pour   des  \'equations diff\'erentielles   stochastiques r\'etrogrades
 avec   des   g\'en\'erateurs  \`a   croissance  quadratique, {C. R. Acad. Sci. Ser. I Math.}, {\bf 324} (1),  (1997), 81--86.
\bibitem{Lep-San}  Lepeltier,  J. P.,  San Martin,  J.,  Backward stochastic differential equations with  continuous coefficients, {Statistic. Probab. Letters}, {\bf 32}, (1997), 425--430.
\bibitem{Mao}  Mao,  X.,   Adapted solution of  Backward stochastic differential equations with  non-Lipschitz coefficients, {Stoch. Proc. Appl}, {\bf 58},  (1997), 281--292.
\bibitem{Par-Peng90} Pardoux,  \'E.,    Peng,   S.,  Adapted  solutions of backward stochastic differential equations,  {Systems and Control Letters}, {\bf 14},  (1997), 55--61.
\bibitem{Par-Peng92} Pardoux, \'E.,   Peng, S.,  Backward stochastic differential equations and quasilinear parabolic PDEs, In: Rozosvskii, B.L., Sowers, R.S. (Eds). Stochastic partial differential equations and their applications, Lect. Notes in Control \& Info. sci., Springer, Berlin, Heidelberg, New York, 176, (1992), 200--217.
\bibitem{Par4} Pardoux,  \'E.,  Generalized discontinuous backward stochastic differential equations. In Backward stochastic differ-
ential equations (Paris, 1995--1996), volume 364 of Pitman Res. Notes Math. Ser., pages 207Ð219. Longman,
Harlow, 1997.
\bibitem{Peng1}  Peng, S.,  Probabilistic interpretation for systems of quasilinear parabolic partial differential equations, {Stochastics Stochastics Reports},  {\bf 37},  (1991), 61--74.
 \bibitem{Royer}  Royer, M.,  Backward stochastic differential equations with  jumps   and  related non-linear  expectations, {Stochastic Processes   and their  Applications},  {\bf 116}, (2006), 1358--1376.
 %
 \bibitem{Tang-Li}  Tang,  S.,  Li,  X., Necessary  conditions   for   optimal control  of   stochastic systems  with  random jumps, {SIAM J. Control Optim.}, {\bf 32}, (5), (1994), 1447--1475. 
\bibitem{Situ} Rong,  S.,  On solutions  of  backward stochastic differential equations with jumps and applications, Stochastic Processes and their Applications, 66, (1997) 209-236.  
\bibitem{Wu}  Wu,  Z., FBSDE with Brownian motion and Poisson Process, Acta Mathematica Applicatae Sinica,  15, No 4, (1999), 433-443.
\bibitem{Wan-Hua}  Wang,  Y.,  Huang,  Z., Backward stochastic differential equations with  non Lipschitz coefficients
equations, {Statistics  and Probability  Letters},  {\bf 79}, (2009), 1438--1443.
\bibitem{Yao}  Yao, S.,  Lp solutions of Backward differential  equation   with  jumps,  Arxiv, 2010.  
\bibitem{Yin-Mao}  Yin, J.,    Mao, X., The adapted solution and comparison theorem for backward stochastic differential equations
with Poisson jumps and applications. J. Math. Anal. Appl., {\bf 346} (2), (2008), 345--358. 
\bibitem{Yin-Situ}   Yin,  J.,  Rong,  S.,  On solutions of forward-backward stochastic differential equations with Poisson jumps. {\em Stochastic Anal. Appl.}, {\bf 21} (6),  (2003), 1419--1448. 
\end{thebibliography}
\end{document}